\documentclass[english,reqno,final]{amsart}   
\usepackage[T1]{fontenc}
\usepackage[latin1]{inputenc}
\usepackage[english]{babel}
\usepackage{amsmath,amsfonts,amssymb,amsthm}
\usepackage[notcite,notref]{showkeys}
\usepackage{mathrsfs}
\usepackage{graphicx}

\newtheorem{theorem}{Theorem}[section]

\newtheorem{proposition}[theorem]{Proposition}
\newtheorem{corollary}[theorem]{Corollary}
\newtheorem*{theorem*}{Theorem}
\theoremstyle{remark}

\newtheorem{definition}[theorem]{Definition}
\newtheorem{example}[theorem]{Example}

\newcommand{\seacsa}{$C^*$-al\-ge\-bra}

\newcommand{\seaFK}{\textnormal{FK}}
\newcommand{\seaPrim}{\textnormal{Prim}}
\newcommand{\seaCK}{Cuntz-Krieger algebra}
\newcommand{\searank}{\textnormal{rank}}
\newcommand{\seasix}{\textnormal{six}}
\newcommand{\seaop}{\textnormal{op}}
\newcommand{\seaFKgunnar}{\seaFK_{\mathcal R}}
\newcommand{\sealookslike}{looks like a \seaCK{}}
\newcommand{\sealooklike}{look like \seaCK{}s}

\input xy
\xyoption{all}

\title{Do phantom Cuntz-Krieger algebras exist?}

\author{Sara Arklint}
\address{Department of Mathematical Sciences, University of Copenhagen, Uni\-versi\-tets\-parken~5, DK-2100 Copenhagen, Denmark}
\email{arklint@math.ku.dk}

\date{\today}

\thanks{This research was supported by the Danish National Research Foundation (DNRF) through the Centre for Symmetry and Deformation, and by the NordForsk Research Network ``Operator Algebras and Dynamics'' (grant \#11580).}

\begin{document}

\begin{abstract}
If  phantom \seaCK s do not exist, then real rank zero \seaCK s can be characterized by outer properties.  In this survey paper, a summary of the known results on non-existence of phantom \seaCK s is given.
\end{abstract}

\maketitle

\section{Introduction}

The Cuntz-Krieger algebras were introduced by J.~Cuntz and W.~Krieger in 1980, cf.~\cite{sea:CK80}, and are a generalization of the Cuntz algebras.
Given an $n \times n$ matrix $A$ with entries in $\{0,1\}$, its associated Cuntz-Krieger algebra $O_{A}$ is defined as the universal \seacsa{} generated by $n$ partial isometries $s_{1},\ldots, s_{n}$ satisfying the relations
\begin{eqnarray*}
1 = s_{1}s_{1}^{*}+\cdots + s_{n}s_{n}^{*} , \\
s_{i}^{*}s_{i} = \sum_{j=1}^{n} A_{ij}s_{j}s_{j}^{*} \textnormal{ for all } i=1,\ldots n .
\end{eqnarray*}
The \seaCK s arise from shifts of finite type, and it has been shown that the \seaCK s are exactly the graph algebras $C^{*}(E)$ arising from finite directed graphs $E$ with no sinks or sources.

Neither of the two equivalent definitions of \seaCK s give an outer characterization of \seaCK s; i.e., neither give a way of determining whether a \seacsa{} is a Cuntz-Krieger algebra, unless it is constructed from a graph or a shift of finite type.

A \seaCK{} is purely infinite if and only if it has real rank zero, and in the following we will mainly restrict to real rank zero \seaCK s since we will rely on classification results that only hold in the purely infinite case.
The \seaCK{} $O_{A}$ is purely infinite if and only if $A$ satisfies Cuntz's condition (II), and equivalently the \seaCK{} $C^{*}(E)$ is purely infinite if and only if the graph $E$ satisfies Krieger's condition (K).  

The notion of \seacsa s over a topological space is useful for defining phantom \seaCK s and for defining filtered $K$-theory, and in~\cite{sea:kirchberg}, Eberhard Kirchberg proved some very powerful classification results for $O_\infty$-absorbing \seacsa s over a space $X$ using $KK(X)$-theory.
A \seacsa{} $A$ over the finite $T_0$-space $X$ is a \seacsa{} equipped with a lattice-preserving map from the open sets of $X$ to the ideals in $A$, denoted $U\mapsto A(U)$ and extended to locally closed subsets as $U\setminus V\mapsto A(U)/A(V)$.
In particular a \seacsa{} with finitely many ideals is a \seacsa{} over its primitive ideal space.
\begin{definition} \label{sea:defphantoms}
A \seacsa{} $A$ with primitive ideal space $X$ \emph{\sealookslike{}} if
\begin{enumerate}
\item $A$ is unital, purely infinite, nuclear, separable, and of real rank zero,
\item $X$ is finite
\item for all $x\in X$, the group $K_*(A(x))$ is finitely generated,
the group $K_1(A(x))$ is free, and \label{sea:pointthree}
$\searank K_0(A(x))=\searank K_1(A(x))$,
\item for all $x\in X$, $A(x)$ is in the bootstrap class of Rosenberg and Schochet. \label{sea:pointfour}
\end{enumerate}
A \seacsa{} that \sealookslike{} but is not isomorphic to a \seaCK{}, is called a \emph{phantom Cuntz-Krieger algebra}.
\end{definition}
All real rank zero \seaCK s \sealooklike{}.
It is not known whether all \seacsa s that \sealooklike{} (and quack like \seaCK s) are \seaCK s.
If it is established that they are, i.e., that phantom Cuntz-Krieger algebras do not exist, then the above definition gives a characterization of the real rank zero \seaCK s.

An example to point out the relevance of such a characterization is given by Proposition~\ref{sea:propextensions}.
If phantom \seaCK s do not exist, the proposition determines exactly when an extension of real rank zero \seaCK s is a real rank zero \seaCK{}.

By a result of Lawrence G.~Brown and Gert K.~Pedersen, Theorem~3.14 of~\cite{sea:brownpedersen}, an extension of real rank zero \seacsa s has real rank zero if and only if projections in the quotient lift to projections in the extension.
Hence, if a \seacsa{} $A$ with primitive ideal space $X$ has real rank zero, then $K_0(A(Y\setminus U))\to K_1(A(U))$ vanishes for all $Y$ and $U$ where $Y$ is a locally closed subsets of $X$ and $U$ is an open subsets of $Y$.
Using this, an induction argument shows that for a \seacsa{} that \sealookslike{}, (\ref{sea:pointthree}) and (\ref{sea:pointfour}) of Definition~\ref{sea:defphantoms} hold for all locally closed subsets $Y$ of $X$.
\begin{proposition} \label{sea:propextensions}
Consider a unital extension $0\to I\to A\to A/I\to 0$ and assume that $A/I$ is a real rank zero \seaCK{} and that $I$ is stably isomorphic to a real rank zero \seaCK{}.
Then $A$ \sealookslike{} if and only if the induced map $K_{0}(A/I)\to K_{1}(I)$ vanishes.
\end{proposition}
\begin{proof}
By Theorem~3.14 of~\cite{sea:brownpedersen}, the \seacsa{} $A$ is of real rank zero if and only if the induced map $K_{0}(A/I)\to K_{1}(I)$ vanishes.
It is well-known or easy to check that the other properties stated in Definition~\ref{sea:defphantoms} are closed under extensions.
\end{proof}

\section{Special cases}

One of the first places one would look for phantom \seaCK s are among the matrix algebras over real rank zero \seaCK s.
Clearly, if $O_A$ is a \seaCK{} of real rank zero, then $M_n(O_A)$ \sealookslike{} for all $n$.
Since $M_n(O_A)$ is a graph algebra, one then immediately asks if a graph algebra can be a phantom \seaCK{}.
It turns out that it cannot.

\begin{theorem}[{\cite{sea:ar}}]
Let $E$ be a directed graph and assume that its graph algebra $C^{*}(E)$ is unital and satisfies $\searank K_{0}(C^{*}(E))=\searank K_{1}(C^{*}(E))$.
Then $C^{*}(E)$ is isomorphic to a \seaCK{}.
\end{theorem}

\begin{theorem}[{\cite{sea:ar}}]
Let $A$ be a unital \seacsa{} and assume that $A$ is stably isomorphic to a \seaCK{}.  Then $A$ is isomorphic to a \seaCK{}.
\end{theorem}

As a small corollary to the work of Eberhard Kirchberg on $KK(X)$-theory, phantom \seaCK s cannot have vanishing $K$-theory.

\begin{theorem}[{\cite{sea:kirchberg}}] \label{sea:thmkirchberg}
Let $A$ and $B$ be unital, nuclear, separable \seacsa s with primitive ideal space $X$.  Then $A\otimes{O_2}$ and $B\otimes{O_2}$ are isomorphic.
\end{theorem}

\begin{corollary}
Let $A$ be a \seacsa{} that \sealookslike{}, and assume that $K_*(A)=0$.  Then $A$ is a Cuntz-Krieger algebra.
\end{corollary}
\begin{proof}
Let $X$ denote the finite primitive ideal space of $A$.  Since $K_*(A)=0$ and $A$ \sealookslike{}, $K_*(A(x))=0$ for all $x\in X$.
So for all $x\in X$, $A(x)$ is  $O_2$-absorbing since it is a UCT Kirchberg algebra with vanishing $K$-theory.  By applying Theorem~4.3 of~\cite{sea:tomswinter} finitely many times, we see that $A$ itself is $O_2$-absorbing.  Let $O_B$ be a Cuntz-Krieger algebra with primitive ideal space $X$ and with $O_B(x)$ (stably) isomorphic to $O_2$ for all $x\in X$.  Then by Theorem~\ref{sea:thmkirchberg}, $A$ is isomorphic to $O_B$.
\end{proof}

\section{Using filtered $K$-theory}

Via $K$-theoretic classification results it can be established that a phantom Cuntz-Krieger algebra cannot have a so-called accordion space as its primitive ideal space.
We will first restrict to the cases where the primitive ideal space has atmost 2 points in order to describe the historical development and due to the importance and powerfulness of the results needed.
The most crucial result is by Eberhard Kirchberg who showed in~\cite{sea:kirchberg} that for stable, purely infinite, nuclear, separable \seacsa s $A$ and $B$ with finite primitive ideal space $X$, any $KK(X)$-equivalence between $A$ and $B$ lift to a $*$-isomorphism.

Simple \seacsa s that \sealooklike{} are UCT Kirchberg algebras, hence the classification result by Eberhard Kirchberg and N.~Christoffer Phillips applies.
For a unital \seacsa{} $A$ with unit $1_A$, denote by $[1_A]$ the class of $1_A$ in $K_0(A)$.
For unital \seacsa s $A$ and $B$ an isomorphism from $(K_*(A),[1_A])$ to $(K_*(B),[1_B])$ is defined as a pair $(\phi_{0},\phi_{1})$ of group isomorphisms $\phi_{i}\colon K_{i}(A)\to K_{i}(B)$, $i=0,1$, for which $\phi_{0}([1_{A}])=\phi_{0}([1_{B}])$.
\begin{theorem}[{\cite{sea:kirchbergphillips}}] \label{sea:thmkirchbergphillips}
Let $A$ and $B$ be unital, simple, purely infinite, nuclear, separable \seacsa s in the bootstrap class.
If $(K_*(A),[1_A])$ and $(K_*(B),[1_B])$ are isomorphic, then $A$ and $B$ are isomorphic.
\end{theorem} 

The range of $K_{*}$ for graph algebras has been determined by Wojciech Szyma\'nski, and his result has been extended by S\o ren Eilers, Takeshi Katsura, Mark Tomforde, and James West to include the class of the unit.

\begin{theorem}[{\cite{sea:ektw}}] \label{sea:thmszymanski}
Let $G$ and $F$ be finitely generated groups, let $g\in G$, and assume that $F$ is free and that $\searank G=\searank F$.
Then there exists a simple \seaCK{} $O_A$ of real rank zero realising $(G\oplus F, g)$ as $(K_*(O_A),[1_{O_A}])$.
\end{theorem}

\begin{corollary}
Simple phantom \seaCK s do not exist.
\end{corollary}
\begin{proof}
Let $A$ be a simple \seacsa{} that \sealookslike{}.  By Theorem~\ref{sea:thmszymanski}, there exists a \seaCK{} $O_B$ of real rank zero for which $(K_*(A),[1_A])\cong(K_*(O_B),[1_{O_B}])$.
Since $A$ and $O_B$ are UCT Kirchberg algebras, it follows from Theorem~\ref{sea:thmkirchbergphillips} that $A$ and $O_B$ are isomorphic.
\end{proof}

For \seacsa s with exactly one nontrivial ideal, the suitable invariant seems to be the induced six-term exact sequence in $K$-theory.  
\begin{definition} 
Let $X_\seasix$ denote the space $\{1,2\}$ with $\{2\}$ open and $\{1\}$ not open.
For a \seacsa{} $A$ with primitive ideal space $X_\seasix$, $K_\seasix(A)$ is defined as the groups and maps
\[ \xymatrix@R=12pt{
K_0(A(2))\ar[r]^-{i} & K_0(A) \ar[r]^-{r} & K_0(A(1)) \ar[d]^-{\delta} \\
K_1(A(1)) \ar[u]^-{\delta} & K_1(A) \ar[l]^-{r} & K_1(A(2)) \ar[l]^-{i} 
} \]
induced by the extension $0\to A(2)\to A\to A(1)$.
For unital \seacsa s $A$ and $B$ with primitive ideal space $X_\seasix$, an isomorphism from $(K_\seasix(A),[1_A])$ to $(K_\seasix(B),[1_B])$ is defined as a triple $(\phi_*^{\{2\}},\phi_*^{X_\seasix},\phi_*^{\{1\}})$ of graded isomorphisms $\phi_*^Y\colon K_*(A(Y))\to K_*(B(Y))$, $Y\in\{\{2\},X_\seasix,\{1\}\}$, that commute with the maps $i$, $r$, and $\delta$ and satisfies $\phi_0^{X_\seasix}([1_A])=[1_B]$.
\end{definition}
This invariant was originally introduced by Mikael R\o rdam to classify \seaCK s with exactly one nontrivial ideal up to stable isomorphism.
Alexander Bonkat established a UCT for $K_\seasix$ (that was later generalized by Ralf Meyer and Ryszard Nest), and by combining his UCT with the result of Eberhard Kirchberg (and a result by Gunnar Restorff and Efren Ruiz in~\cite{sea:rr} to achieve unital and not stable isomorphism) one obtains the following theorem.

\begin{theorem}[{\cite{sea:bonkat,sea:kirchberg}}]
Let $A$ and $B$ be unital, purely infinite, nuclear, separable \seacsa s with primitive ideal space $X_\seasix$, and assume that $A(x)$ and $B(x)$ are in the bootstrap class for all $x\in\{1,2\}$.
Then $(K_\seasix(A),[1_A])\cong(K_\seasix(B),[1_B])$ implies $A\cong B$.
\end{theorem}

The range of $K_\seasix$ for graph algebras has been determined by S\o ren Eilers, Takeshi Katsura, Mark Tomforde, and James West.

\begin{theorem}[{\cite{sea:ektw}}] \label{sea:thmektw}
Let a six-term exact sequence 
\[ \xymatrix@R=4pt{
&G_1\ar[r] & G_2 \ar[r] & G_3 \ar[dd]^-{0} & \\
\mathcal E : && & \\
&F_3 \ar[uu] & F_2 \ar[l] & F_1 . \ar[l] & 
} \]
be given with $G_1,G_2,G_3$ and $F_1,F_2,F_3$  finitely generated groups, and let $g\in G_2$.
Assume that the groups $F_1,F_2,F_3$ are free, and that $\searank G_i=\searank F_i$ for all $i=1,2,3$.
Then there exists a \seaCK{} $O_A$ of real rank zero with primitive ideal space $X_\seasix$ realising $(\mathcal E,g)$ as $(K_\seasix(O_A),[1_{O_A}])$.
\end{theorem}

\begin{corollary}
Phantom \seaCK s with exactly one nontrivial ideal do not exist.
\end{corollary}

The generalization of the invariant $K_\seasix$ to larger primitive ideal spaces is called filtered $K$-theory or filtrated $K$-theory and was introduced by Gunnar Restorff and by Ralf Meyer and Ryszard Nest.
Filtered $K$-theory consists of the six-term exact sequences induced by all extensions of subquotients.
A smaller invariant, the reduced filtered $K$-theory $\seaFKgunnar$ originally defined by Gunnar Restorff to classify \seaCK s, has so far proven suitable for classifying \seacsa s that \sealooklike{}.

Let $X$ be a finite $T_0$-space. For $x\in X$, we denote by $\widetilde{\{x\}}$ the smallest open subset of $X$ containing $x$, and we define $\widetilde{\partial}(x)$ as $\widetilde{\{x\}}\setminus \{x\}$.
For $x,y\in X$ we write $y\to x$ when $y\in\widetilde\partial(x)$ and there is no $z\in\widetilde\partial(x)$ for which $y\in\widetilde\partial(z)$.
\begin{definition}
For a \seacsa{} $A$ with primitive ideal space $X$, its {\emph{reduced filtered $K$-theory} $\seaFKgunnar(A)$} consists of the groups and maps
\[ \xymatrix{
K_1(A(x)) \ar[r]^-{\delta} & K_0(A(\widetilde\partial(x))) \ar[r]^-{i} & K_0(A(\widetilde{\{x\}})) 
} \]
induced by the extension $0\to A(\widetilde\partial(x))\to A(\widetilde{\{x\}})\to A(x)\to 0$,
for all $x\in X$, together with the groups and maps
\[ \xymatrix{
K_0(A(\widetilde{\{y\}}) \ar[r]^-{i} & K_0(A(\widetilde\partial(x)))
} \]
induced by the extension $0\to A(\widetilde{\{y\}})\to A(\widetilde\partial(x))\to A(\widetilde\partial(x)\setminus\widetilde{\{y\}})\to 0$, for all $x,y\in X$ with $y\to x$.
\end{definition}

\begin{example}
Let $X=\{1,2,3\}$ be given the topology $\{\emptyset, \{3\}, \{3,2\}, \{3,1\},X\}$.  Then for a \seacsa{} $A$ with primitive ideal space $X$, its reduced filtered $K$-theory $\seaFKgunnar(A)$ consists the groups and maps
\[ \xymatrix@R=12pt@C=12pt{
K_1(A(2))\ar[dr]^-{\delta} && K_0(A(\{3,1\})) \\
 &K_0(A(3))\ar[ur]^-{i}\ar[dr]^-{i}& \\
K_1(A(1))\ar[ur]^-{\delta} && K_0(A(\{3,2\})) 
} \]
together with the group $K_1(A(3))$.
\end{example}

It is shown in~\cite{sea:abk} that if $A$ is a \seacsa{} of real rank zero with primitive ideal space $X$, then the sequence
\[
\bigoplus_{y\to x, y\to x'} K_0(A(\widetilde{\{y\}})) \stackrel{(i^2  \: -i^2)}{\longrightarrow} \bigoplus_{x\in X} K_0(A(\widetilde{\{x\}})) \stackrel{(i)}{\longrightarrow} K_0(A) \longrightarrow 0
\]
is exact.

\begin{definition}
For a unital \seacsa{} $A$ of real rank zero with primitive ideal space $X$, $1(A)$ is defined as the unique element in
\[ \bigoplus_{x\in X} K_0(A(\widetilde{\{x\}}))  / \bigoplus_{y\to x, y\to x'} K_0(A(\widetilde{\{y\}})) \]
that is mapped to $[1_A]$.
For $A$ and $B$ unital \seacsa s of real rank zero with primitive ideal space $X$, an isomorphism from $(\seaFKgunnar(A),1(A))$ to $(\seaFKgunnar(B),1(B))$ is defined as a family of isomorphisms
\begin{eqnarray*}
\phi_{\{x\}} \colon K_1(A(x))\to K_1(B(x))\\
\phi_{\widetilde\partial(x)} \colon K_0(A(\widetilde\partial(x))\to K_0(B(\widetilde\partial(x))\\
\phi_{\widetilde{\{x\}}} \colon K_0(A(\widetilde{\{x\}}))\to  K_0(B(\widetilde{\{x\}}))
\end{eqnarray*}
for all $x\in X$ that commute with the maps $i$ and $\delta$ and maps $1(A)$ to $1(B)$.
\end{definition}

Using Theorem~\ref{sea:thmektw}, Rasmus Bentmann, Takeshi Katsura, and the author have established the range of reduced filtered $K$-theory $\seaFKgunnar$ for graph algebras.

\begin{theorem}[{\cite{sea:abk}}]
Let $B$ be a \seacsa{} that \sealookslike{}.  Then there exists a \seaCK{} $O_A$ of real rank zero with $\seaPrim(O_A)\cong\seaPrim(B)$ for which $(\seaFKgunnar(O_A),[1_{O_A}])$ is isomorphic to $(\seaFKgunnar(B),[1_B])$.
\end{theorem}

\begin{definition} \label{sea:defaccordion}
A finite, connected $T_0$-space $X$ is called an \emph{accordion space} if for all $x\in X$ there are at most two elements $y\in X$ for which $y\to x$, and if there is at least two elements $x\in X$ for which there is exactly one element $y\in X$ for which $y\to x$.
\end{definition}

The notion of accordion spaces was introduced by Rasmus Bentmann in~\cite{sea:bentmann}.
Intuitively, a space is an accordion space if and only if the Hasse diagram of the ordering defined by $y\leq x$ when $y\in\widetilde{\{x\}}$, looks like an accordion.
All finite, linear spaces are accordion spaces, and the following five spaces are examples of connected spaces that are not accordion spaces.

\begin{definition} \label{sea:deffourpointspaces}
Define a topology on the space $\mathcal X=\{1,2,3,4\}$ by defining $U\subseteq\mathcal X$ to be open if $U$ is empty or $4\in U$.  Define $\mathcal X^\seaop$ as having the opposite topology.  Then $\mathcal X$ and $\mathcal X^\seaop$ have Hasse diagrams
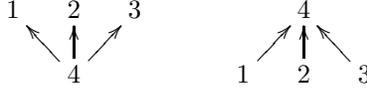
\[ \xymatrix@R=12pt@C=12pt{ 
1 & 2 & 3 && &4& \\
&4\ar[ul]\ar[u]\ar[ur]& && 1\ar[ur] & 2\ar[u] & 3\ar[ul]
} \]
respectively.
Define a topology on the space $\mathcal Y=\{1,2,3,4\}$ by defining $U\subseteq\mathcal X$ to be open if $U\in\{\emptyset,\{4\}\}$ or if $\{3,4\}\subseteq U$.  Define $\mathcal Y^\seaop$ as having the opposite topology.  Then $\mathcal Y$ and $\mathcal Y^\seaop$ have Hasse diagrams
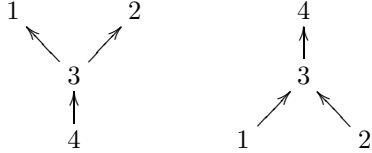
\[ \xymatrix@R=12pt@C=12pt{ 1 && 2 && &4& \\
&3\ar[ul]\ar[ur]& && &3\ar[u]& \\
&4\ar[u]& && 1\ar[ur]&&2\ar[ul]
} \]
respectively.
Finally, define a topology on the space $\mathcal D=\{1,2,3.4\}$ as the open sets being $\{\emptyset,\{4\},\{3,4\},\{2,4\},\{2,3,4\},\mathcal D\}$. Then $\mathcal D$ has Hasse diagram
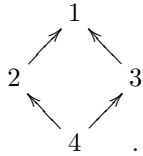
\[ \xymatrix@R=12pt@C=12pt{
& 1 & \\
2\ar[ur] && 3\ar[ul] \\
& 4\ar[ur]\ar[ul] & .
} \]
\end{definition}

Ralf Meyer and Ryszard Nest showed in~\cite{sea:meyernest} that if $X$ is a finite, linear space, then filtered $K$-theory is a complete invariant for all stable, purely infinite, nuclear, separable \seacsa s $A$ with primitive ideal space $X$ that satisfy that $A(x)$ are in the bootstrap class for all $x\in X$.  They also gave a counter-example to completeness of filtered $K$-theory for the space $\mathcal X$.  Using their methods, Rasmus Bentmann and Manuel K\"ohler showed in~\cite{sea:bk} that filtered $K$-theory is a complete invariant for such \seacsa s exactly when their primitive ideal space $X$ is an accordion space.

However, Gunnar Restorff, Efren Ruiz, and the author showed in~\cite{sea:arr} that for the spaces $\mathcal X$, $\mathcal X^\seaop$, $\mathcal Y$, and $\mathcal Y^\seaop$, filtered $K$-theory is a complete invariant for such \seacsa s if one adds the assumption of real rank zero.  And in~\cite{sea:abk}, Rasmus Bentmann, Takeshi Katsura, and the author showed that for the space $\mathcal D$, reduced filtered $K$-theory is a complete invariant for \seacsa s that \sealooklike{}.  It is also shown in~\cite{sea:abk} that for \seacsa s that \sealooklike{} and have either an accordion space or one of the spaces defined in Definition~\ref{sea:deffourpointspaces} as primitive ideal space, any isomorphism on reduced filtered $K$-theory can be lifted to an isomorphism on filtered $K$-theory.

The five spaces of Definition~\ref{sea:deffourpointspaces} are so far the only non-accordion spaces for which such results have been achieved.
However, combining these results with Theorem~2.1 of~\cite{sea:rr} gives the following theorem, cf.~\cite{sea:abk}.

\begin{theorem}[{\cite{sea:meyernest,sea:bentmann,sea:arr,sea:abk}}]
Let $X$ be either an accordion space or one of the spaces defined in Definition~\ref{sea:deffourpointspaces}.
Let $A$ and $B$ be \seacsa s that \sealooklike{} and have $X$ as primitive ideal space.
Then $(\seaFKgunnar(A),1(A))\cong(\seaFKgunnar(B),1(B))$ implies $A\cong B$.
\end{theorem}

\begin{corollary}
Let $X$ be either an accordion space or one of the spaces defined in Definition~\ref{sea:deffourpointspaces}.
Then phantom \seaCK s with primitive ideal space $X$ do not exist.
\end{corollary}

\section{Summary}
The results stated in this article, are recaptured in the following theorem.
\begin{theorem}
Let $A$ be a \seacsa{} that \sealookslike{}.  If $A$ satisfies either of the following conditions, 
\begin{itemize}
\item $A$ is a graph algebra,
\item $K_{*}(A)=0$,
\item $\seaPrim(A)$ is an accordion space,
\item $\seaPrim(A)$ is one of the five four-point spaces of Definition~\ref{sea:deffourpointspaces},
\end{itemize}
then $A$ is isomorphic to a \seaCK{}.
\end{theorem}

It is unknown whether phantom \seaCK s exist in general.

\end{document}